\pdfoutput=1
\RequirePackage{ifpdf}
\ifpdf 
\documentclass[pdftex]{sigma}
\else
\documentclass{sigma}
\fi

\newcommand{\bC}{\mathbb{C}}
\newcommand{\bM}{\mathbb{M}}
\newcommand{\bQ}{\mathbb{Q}}
\newcommand{\bR}{\mathbb{R}}
\newcommand{\bZ}{\mathbb{Z}}
\newcommand{\fH}{\mathfrak{H}}
\newcommand{\Tr}{\operatorname{Tr}}

\newtheorem{thm}{Theorem}[section]
\newtheorem*{conj}{Conjecture}
\newtheorem{prop}[thm]{Proposition}

\newtheorem{cor}[thm]{Corollary}
\newtheorem*{thmm}{Theorem}

\theoremstyle{definition}
\newtheorem{defn}[thm]{Definition}
\newtheorem{rem}[thm]{Remark}

\begin{document}
\allowdisplaybreaks

\newcommand{\arXivNumber}{1712.10160}

\renewcommand{\thefootnote}{}

\renewcommand{\PaperNumber}{114}

\FirstPageHeading

\ShortArticleName{Characterizing Moonshine Functions by Vertex-Operator-Algebraic Conditions}

\ArticleName{Characterizing Moonshine Functions\\ by Vertex-Operator-Algebraic Conditions\footnote{This paper is a~contribution to the Special Issue on Modular Forms and String Theory in honor of Noriko Yui. The full collection is available at \href{http://www.emis.de/journals/SIGMA/modular-forms.html}{http://www.emis.de/journals/SIGMA/modular-forms.html}}}

\Author{Scott CARNAHAN, Takahiro KOMURO and Satoru URANO}
\AuthorNameForHeading{S.~Carnahan, T.~Komuro and S.~Urano}
\Address{Division of Mathematics, University of Tsukuba,\\
 1-1-1 Tennodai, Tsukuba, Ibaraki 305-8571 Japan}
\Email{\href{mailto:carnahan@math.tsukuba.ac.jp}{carnahan@math.tsukuba.ac.jp}, \href{mailto:komenian271828@gmail.com}{komenian271828@gmail.com}, \href{mailto:urank@math.tsukuba.ac.jp}{urank@math.tsukuba.ac.jp}}
\URLaddress{\url{http://www.math.tsukuba.ac.jp/~carnahan/}}

\ArticleDates{Received May 07, 2018, in final form October 15, 2018; Published online October 25, 2018}

\Abstract{Given a holomorphic $C_2$-cofinite vertex operator algebra $V$ with graded dimension $j-744$, Borcherds's proof of the monstrous moonshine conjecture implies any finite order automorphism of $V$ has graded trace given by a ``completely replicable function'', and by work of Cummins and Gannon, these functions are principal moduli of genus zero modular groups. The action of the monster simple group on the monster vertex operator algebra produces 171 such functions, known as the monstrous moonshine functions. We show that 154 of the 157 non-monstrous completely replicable functions cannot possibly occur as trace functions on $V$.}


\Classification{11F22; 17B69}

\renewcommand{\thefootnote}{\arabic{footnote}}
\setcounter{footnote}{0}

\section{Introduction}

In this paper, we show that the monstrous moonshine functions are nearly completely charac\-terized by vertex-operator-algebra-theoretic properties of trace functions. In particular, we strongly constrain the possible exotic automorphisms of any counterexample to the conjectured uniqueness of the monster vertex operator algebra $V^\natural$ (also known as the moonshine module) constructed in~\cite{FLM88}. Our constraints are derived from the following two recent vertex-operator-algebraic results:
\begin{enumerate}\itemsep=0pt
\item In \cite{vEMS}, it is shown that the graded dimension of an irreducible $g$-twisted module (of a~suitable vertex operator algebra) is given by applying the $\tau \mapsto -1/\tau$ transformation to the McKay--Thompson series for $g$. This is a refinement of an older result of~\cite{DLM97}.
\item In \cite{C17}, we find that cyclic orbifold duality gives a correspondence between non-Fricke elements of the monster and fixed-point free automorphisms of the Leech lattice satisfying a ``no massless states'' condition. This was conjectured by Tuite in~\cite{T93}.
\end{enumerate}

Monstrous moonshine originated in the 1970s with McKay's observation that the low-degree Fourier coefficients of the modular $j$-invariant could be decomposed into small combinations of dimensions of irreducible representations of the monster simple group~$\bM$. This initial observation formed the basis of what is known as the McKay--Thompson conjecture, asserting the existence of a natural graded representation $V = \bigoplus\limits_{n=0}^\infty V_n$ of the monster, whose graded dimension is given by $\sum\limits_{n =0}^\infty \dim V_n q^{n-1} = J(\tau) = q^{-1} + 196884q + \cdots$. Conway and Norton formulated the monstrous moonshine conjecture as a refinement, asserting that the graded trace of any element $g \in \bM$ should be the $q$-expansion of a genus zero modular function $T_g$, and computed a list of candidate functions.

The McKay--Thompson conjecture was affirmatively resolved by Frenkel, Lepowsky, and Meurman with their construction of $V^\natural$, and Borcherds showed that the monstrous moonshine conjecture was true for~$V^\natural$~\cite{B92}. The main body of Borcherds's proof was devoted to showing that for each $g \in \bM$, the graded trace function
\begin{gather*} T_g(\tau) = \sum_{n=0}^\infty \Tr(g|V_n) q^{n-1} \end{gather*}
is completely replicable. The completely replicable condition implies strong recursion relations on the coefficients of $T_g$, and indeed, these relations reduced the monstrous moonshine conjecture to checking that the first 7 coefficients of each $T_g$ matched the candidate list given by Conway and Norton.

Borcherds's proof of the completely replicable property did not use very much structure in~$V^\natural$: he only required that it be a vertex operator algebra of central charge 24, with a self-dual invariant form, and graded dimension equal to~$J$. It is natural to ask whether it is possible for other completely replicable functions to arise from other vertex operator algebras $V$ satisfying these properties. At our current state of knowledge, we cannot answer this question, because we do not have much control over this class of vertex operator algebras. However, we can make some progress by imposing additional natural conditions on $V$ that are also satisfied by $V^\natural$.

The additional condition we impose in this paper is that $V$ be ``holomorphic and $C_2$-cofinite'', which is equivalent to the condition that all (weak) representations of $V$ are direct sums of $V$ itself. A well-known uniqueness conjecture proposed in the introduction of~\cite{FLM88} implies any such~$V$ is isomorphic to~$V^\natural$. In particular, assuming this uniqueness conjecture, it is clear that the graded traces from $V^\natural$ make up all possible completely replicable functions arising from graded traces on such~$V$. However, the uniqueness conjecture is still open, so it is still an interesting question whether we can narrow down the possible functions using facts we already know about vertex operator algebras. In doing so, we also narrow down the possible symmetries of any counterexample to the uniqueness conjecture, assuming such a counterexample exists.

The ``holomorphic and $C_2$-cofinite'' condition allows us to connect $V$ with the phenomena of rational conformal field theory, and in particular, it gives us control over the twisted representations of~$V$. The pioneering work of \cite{DLM97} showed that if $g$ is a finite order automorphism of such~$V$, then there exists an irreducible $g$-twisted $V$-module $V(g)$, and it is unique up to isomorphism. Furthermore, they showed that the graded dimension of $V(g)$ is a constant multiple of $T_g(-1/\tau)$, where $T_g$ is the graded trace of $g$ on $V$. This constant multiple was then shown to be equal to one in~\cite{vEMS}. We then obtain a nontrivial condition on $T_g$: because dimensions of vector spaces are non-negative integers, it is necessary that the coefficients of $T_g(-1/\tau)$ are all non-negative integers. As it happens, there is a refined version of this condition, involving non-negativity of coefficients of a certain vector-valued modular function, that follows from the cyclic orbifold theory established in~\cite{vEMS}. However, this refined condition has not helped, in the sense that it does not eliminate any functions that couldn't be eliminated by less computationally-intensive methods.

More recent developments give us an additional condition: In \cite{T93}, a large amount of computational and physical evidence is given for a conjectured cyclic orbifold duality between non-Fricke automorphisms of $V^\natural$ and fixed-point free automorphisms of the Leech lattice satisfying a ``no massless states'' condition. Using the cyclic orbifold theory established in~\cite{vEMS}, the conjecture was proved in~\cite{C17}, and it implies that any completely replicable function with a non-Fricke monstrous replicate must be monstrous in order to be a trace function. We note that this implies any counterexample to the uniqueness of~$V^\natural$ has no non-Fricke automorphisms.

These two conditions are sufficient to eliminate all but three of the non-monstrous functions from consideration, yielding our main theorem (Theorem \ref{thm:main}):

\begin{thmm}All primitive non-monstrous completely replicable functions except possibly $9a$ are non-moonshine-like. All non-monstrous completely replicable functions except possibly $9a$, $63a$, and $117a$ are non-moonshine-like.
\end{thmm}

The remaining question is to find a reason to eliminate the function~9a, because the elimination of the other two functions, 63a and 117a, would follow automatically. We hope that a~thorough investigation of this question will yield some new condition that will eliminate this function and thereby give a characterization of the monstrous moonshine functions by vertex-operator-algebraic conditions. Some ideas from string theory, in particular the description of Atkin--Lehner involutions in terms of dualities of CHL dyon models given in~\cite{PPV16} and~\cite{PPV17} look potentially promising, but we have been unable to extract a strong argument.

There have been other attempts to characterize the monstrous moonshine functions in terms of conditions on their invariance groups, starting with an incomplete characterization in~\cite{CN79}, where there were three extra ``ghost'' functions 25Z, 49Z, and 50Z (now called 25a, 49a, and~50a). This was followed by complete characterizations given in~\cite{CMS04} and~\cite{DF09}. However, we still do not know a way to completely connect the conditions given in these papers to properties satisfied by all possible vertex operator algebras similar to $V^\natural$. In particular, all three of these papers use the condition that the fixing groups of functions are contained in the involutory normalizer of $\Gamma_0(N)$. That is, for any $\gamma \in {\rm SL}_2(\bR)$ fixing a suitable function, one assumes that $\gamma^2 \in \Gamma_0(N)$. If we could find a vertex-operator-algebra-theoretic reason why this condition is necessary, we could eliminate the three remaining functions.

\section{Overview: completely replicable functions}

Given a holomorphic function $f(\tau) = q^{-1} + a_1 q + a_2 q^2 + \cdots$ (with $q = e^{2\pi {\rm i} \tau}$) on the complex upper half-plane with $a_i \in \bZ$, and a positive integer $n$, there is a unique polynomial $\Phi_{f,n}(x) = x^n + c_{n-1}x^{n-1} + \cdots + c_0$ such that $\Phi_{f,n}(f(\tau))$ has $q$-expansion of the form $q^{-n} + O(q)$. The polynomial $\Phi_{f,n}$ is called the normalized Faber polynomial of $f$.

We say that $f$ is replicable if there are functions $f^{(n)}$ of the form $q^{-1} + O(q)$ for all $n \geq 1$, such that $f = f^{(1)}$, and
\begin{gather*} \Phi_{f,n}(f(\tau)) = \sum_{ad=n,\, 0 \leq b < d} f^{(a)}\left(\frac{a\tau + b}{d}\right). \end{gather*}
As noted in \cite{F96}, the right side is a modification of the weight~0 Hecke operator. If $f$ is replicable, we call the function $f^{(n)}$ the $n$-th replicate of $f$. We say that $f$ is completely replicable if it is replicable and all $f^{(n)}$ are replicable.

The completely replicable functions have been classified in the following sense. A list of candidate functions was computed in \cite{ACMS92}, and this list turned out to be complete by virtue of later results, which we briefly summarize. In \cite{CG97}, the main theorem is that all non-degenerate completely replicable functions are Hauptmoduln, that is, they are weight 0 modular functions $f\colon \fH \to \bC$ invariant under a discrete subgroup $\Gamma_f < {\rm SL}_2(\bR)$ containing some $\Gamma_0(N)$, such that the induced map $\Gamma_f \backslash \fH \to \bC$ is an open immersion with dense image. Finally, \cite{C04} contains an exhaustive enumeration of genus zero modular groups. There are 328 completely replicable functions, 325 of which are Hauptmoduln, and the three ``modular fictions'' are $q^{-1}$, $q^{-1} + q$ and $q^{-1} - q$.

Completely replicable functions are intimately connected to moonshine because of the following result of \cite{B92}:

\begin{thm}
Let $V = \bigoplus\limits_{n=0}^\infty V_n$ be a vertex operator algebra of central charge $24$, with graded dimension given by the modular invariant function $J(\tau) = j(\tau) - 744 = q^{-1} + 196884q + \cdots$. Suppose $V$ admits an invariant bilinear form for which it is self-dual. Then for any finite order automorphism~$g$ of~$V$, the graded trace
\begin{gather*} T_g(\tau) = \sum_{n =0}^\infty \Tr(g|V_n)q^{n-1} \end{gather*}
is a completely replicable function, and for any $k \geq 1$, its $k$-th replicate is~$T_{g^k}$.
\end{thm}

As we mentioned in the introduction, we don't have much control over vertex operator algebras satisfying only the hypotheses of this theorem. However, $V^\natural$ is also $C_2$-cofinite by~\cite{DLM97} and holomorphic by~\cite{D94}. We therefore consider the following definitions:
\begin{defn}\quad
\begin{enumerate}\itemsep=0pt
\item A vertex operator algebra is ``moonshine-like'' if it is $C_2$-cofinite, holomorphic, and its graded dimension is given by the modular invariant function $J(\tau) = j(\tau) - 744 = q^{-1} + 196884q + \cdots$.
\item A completely replicable function is ``moonshine-like'' if it is given by the graded trace of a finite order automorphism on a moonshine-like vertex operator algebra. Otherwise, we say the function is ``non-moonshine-like''.
\end{enumerate}
\end{defn}

\begin{conj}For any moonshine-like function $f$, there is some $g \in \bM$ such that $f = T_g$.
\end{conj}

Our goal, which we do not quite achieve, is to show the contrapositive, i.e., that all non-monstrous functions are non-moonshine-like. The previous theorem immediately implies the following:

\begin{cor} \label{cor:bad-if-replicate-is-bad}Let $f$ be a completely replicable function. If $f^{(k)}$ is non-moonshine-like for some~$k$, then~$f$ is also non-moonshine-like.
\end{cor}

\begin{rem}This allows us to reduce our calculations substantially. In order to show that all non-monstrous functions are non-moonshine-like, it suffices to eliminate those whose non-trivial replicates are monstrous. We will call these functions primitive non-monstrous functions. The full list of these is (in the notation of \cite{ACMS92}): 2a, 4a, 5a, 6d, 8a, 8b, 8c, 9a, 9b, 9c, 9d, 12a, 12c, 12d, 15a, 16a, 16b, 16c, 16d, 16e, 16f, 20d, 24f, 24g, 24h, 25a, 27a, 27b, 27c, 32b, 36f, 40d, 44a, 44b, 49a. In the remainder of this paper, we show that all of these functions except 9a are non-moonshine-like. To determine precisely what functions our methods can prove to be non-moonshine-like, we must also consider functions that have 9a as a replicate. We show that 18e and 18h are non-moonshine-like, but our methods don't work on 63a or 117a.
\end{rem}

\section{Condition I: positivity for twisted modules}

\begin{thm}[\cite{DLM97}]For any holomorphic, $C_2$-cofinite vertex operator algebra $V$ with a finite order automorphism $g$, there is an irreducible $g$-twisted module $V(g)$, unique up to isomorphism.
\end{thm}

\begin{thm}[\cite{vEMS}]The graded dimension of $V(g)$ is given by the $q$-expansion of $T_g(-1/\tau)$, where $T_g(\tau) = \sum\limits_{n = 0}^\infty \Tr(g|V_n) q^{n-c/24}$, and $c$ is the central charge.
\end{thm}

\begin{cor} \label{cor:non-negative-integer-coefficient}If the $q$-expansion of $f(-1/\tau)$ has a coefficient that is not a non-negative integer, then $f$ is non-moonshine-like.
\end{cor}

Thus, we may work out our first explicit examples.

\begin{prop} \label{prop:fricke-involution-condition}Let $f$ be a completely replicable function invariant under $\Gamma_0(N)$, such that either
\begin{enumerate}\itemsep=0pt
\item[$1)$] the Fricke involution $W_N$ acts by $-1$, or
\item[$2)$] the Fricke involution $W_N$ acts by identity, and $f$ has negative integer coefficients.
\end{enumerate}
Then $f$ is non-moonshine-like.
\end{prop}
\begin{proof}For the first case, because the class $W_N$ contains the transformation $\tau \mapsto \frac{-1}{N\tau}$, we have $f\big({-}\frac{1}{\tau}\big) = -f\big(\frac{\tau}{N}\big) = -q^{-1/N} + O(1)$, which has a negative integer coefficient.

For the second case, suppose $f$ has a negative integer coefficient $a_k$ attached to $q^k$. Then applying the hypothesis on $W_N$ yields $f(-\frac{1}{\tau}) = f(\frac{\tau}{N})$, so the coefficient on $q^{k/N}$ is the negative integer~$a_k$.

In either case, Corollary \ref{cor:non-negative-integer-coefficient} applies, and we find that the function $f$ is non-moonshine-like.
\end{proof}

We would like to know which completely replicable functions are eigenfunctions for a Fricke involution. For those of type $n|h$, this information may be read off the tables in Section~1.3 of~\cite{F93}. The information given is interpreted as follows: if $n/h$ is even, then the notation $n|h + \overline{n/h}$ means $W_N$ has eigenvalue $-1$, and if $n$ is even but $n/h$ is odd, then the notation $n|h + n/h$ means $W_N$ has eigenvalue $-1$.

\begin{cor} \label{cor:bad-fricke-list}The primitive non-monstrous completely replicable functions $2a$, $4a$, $8a$, $12a$, and $44a$ are non-moonshine-like.
\end{cor}
\begin{proof}These are precisely the primitive non-monstrous $n|h$-type functions for which $W_N$ acts as $-1$, so Proposition~\ref{prop:fricke-involution-condition} applies.
\end{proof}

\looseness=1 For the other functions in \cite{ACMS92}, we need to work a bit harder to determine which ones have non-moonshine-like Fricke behavior. However, we can use the enumeration of cusps given in~\cite{C04}.

\begin{prop} \label{prop:one-cusp}Let $f$ be a completely replicable function, such that the coefficients of $f(a\tau+b)$ are not all non-negative integers, for any $a \in \bQ^\times$ and $b \in \bQ$. Furthermore, suppose the fixing group of $f$ has only one orbit of cusps. Then $f$ is non-moonshine-like.
\end{prop}
\begin{proof}Because all cusps are equivalent under the fixing group of $f$, we see that $f(-1/\tau)$ is equal to $f(a\tau+b)$ for some $a \in \bQ^\times$ and $b \in \bQ$. By our hypothesis about the coefficients of $f(a\tau+b)$, the conditions of Corollary \ref{cor:non-negative-integer-coefficient} are satisfied.
\end{proof}

\begin{cor} \label{cor:bad-one-cusp-list}The primitive non-monstrous functions $5a$, $6d$, $8b$, $8c$, $9b$, $9d$, $15a$, $16e$, $16f$, $24f$, and $24g$ are non-moonshine-like. The non-primitive function $18h$ is also non-moonshine-like.
\end{cor}
\begin{proof}We show that these functions satisfy the criteria of Proposition \ref{prop:one-cusp}. The fact that the relevant fixing groups have one orbit of cusps was computed in \cite{C04}. Specifically, the width and number of cusps is given in Table~2 of~\cite{C04}, and the completely replicable functions are given in Table 3, so we reproduce the names of the groups in the two tables, together with the cusp width information:

\begin{center}
\begin{tabular}{c|c|c|c}
$f$ & Table 3 Name & Table 2 Name & cusps \\ \hline
5a & $25A^0$ & $5A^0_1$ & $5^1$ \\
6d & $18B^0$ & $3A^0_2$ & $3^1$\\
8b & $32B^0$ & $4B^0_2$ & $4^1$ \\
8c & $32B^0$ & $4B^0_2$ & $4^1$ \\
9b & $27A^0$ & $3B^0_3$ & $3^1$ \\
9d & $81A^0$ & $9A^0_1$ & $9^1$ \\
15a & $45A^0$ & $3A^0_5$ & $3^1$ \\
16e & $128A^0$ & $8G^0_2$ & $8^1$ \\
16f & $128A^0$ & $8G^0_2$ & $8^1$ \\
18h & $54B^0$ & $3B^0_6$ & $3^1$ \\
24f & $96A^0$ & $4A^0_6$ & $4^1$ \\
24g & $96A^0$ & $4A^0_6$ & $4^1$
\end{tabular}
\end{center}

For the claim about coefficients of $f(a\tau+b)$ it suffices to show that the signs of coefficients of $f$ are neither constant nor uniformly alternating, and this can be checked straightforwardly from the tables in \cite{N} or a short computation using the first few terms given in~\cite{ACMS92}.
\end{proof}

\section{Condition II: cyclic orbifold duality}

One of the main results of \cite{vEMS} is that if we are given a holomorphic $C_2$-cofinite vertex operator algebra $V$ of CFT type, and a finite order automorphism $g$, then there exists a cyclic orbifold dual $V/g$, which is also a holomorphic $C_2$-cofinite vertex operator algebra of CFT type, and it is equipped with a dual automorphism $g^*$ of the same order, such that $(V/g)/g^* \cong V$ and $g^{**} = g$. We will use this duality to eliminate certain completely replicable functions from consideration.

In \cite{C17}, the first author proved a conjecture in \cite{T93}, giving an orbifold-duality correspondence between non-Fricke automorphisms of $V^\natural$ and fixed-point free automorphisms of the Leech lattice satisfying a ``no massless states'' condition.

\begin{prop}Let $V$ be a moonshine-like vertex operator algebra, and let $g$ be an automorphism such that $T_{g^k}$ is a non-Fricke monstrous function for some $k$. Then $V \cong V^\natural$, and $g \in \bM$.
\end{prop}
\begin{proof}By our assumption that $T_{g^k}$ is a non-Fricke monstrous function, the proof of Theorem~4.5 of~\cite{C17} implies $V/g^k$ is isomorphic to $V_\Lambda$. We note that the statement of the theorem assumes \smash{$V \cong V^\natural$}, but the proof only uses the completely replicable function $T_{g^k}$ and its replicates and ${\rm SL}_2(\bZ)$-transforms, so it works under the assumption that $V$ is only moonshine-like. By Corollary~4.6 of~\textit{loc.~cit.}, cyclic orbifold duality then implies $V \cong V^\natural$, since $(g^k)^*$ is then a~fixed-point free anomaly-free automorphism of $\Lambda$ satisfying the ``no massless states condition'' given in~\cite{T93}.
\end{proof}

\begin{cor} \label{cor:bad-orbifold-list} The primitive non-monstrous functions $4a$, $9c$, $12c$, $12d$, $16a$, $16b$, $16c$, $16d$, $20d$, $24h$, $25a$, $27a$, $27b$, $27c$, $32b$, $36f$, $40d$, $44b$, and $49a$ are non-moonshine-like. The non-primitive function $18e$ is also non-moonshine-like.
\end{cor}
\begin{proof}In each case, there is some $k > 1$ such that the $k$th replicate is a non-Fricke monstrous function.
\end{proof}

\section{The last three functions}

We now summarize what we know:

\begin{thm} \label{thm:main}All primitive non-monstrous functions except possibly $9a$ are non-moonshine-like. All non-monstrous functions except possibly $9a$, $63a$, and $117a$ are non-moonshine-like.
\end{thm}
\begin{proof}The claim about primitive functions follows from Corollaries \ref{cor:bad-fricke-list}, \ref{cor:bad-one-cusp-list}, and~\ref{cor:bad-orbifold-list}. 9a is a~replicate of the functions 18b, 18e, 18h, 36b, 36c, 36e, 36h, 45c, 63a, 72b, 117a, and 126a, so we need to check these to settle the second claim. However, 18e is non-moonshine-like by Corollary~\ref{cor:bad-orbifold-list} and 18h is non-moonshine-like by Corollary~\ref{cor:bad-one-cusp-list}. Of the remaining functions, we find that all but 63a and 117a have non-moonshine-like replicates, so are non-moonshine-like by Corollary~\ref{cor:bad-if-replicate-is-bad}.
\end{proof}

We would have a complete characterization if the following were true:

\begin{conj}The completely replicable function $f_{9a} = q^{-1} + 14q^2 + 65q^5 + 156q^8 + \cdots$ is non-moonshine-like.
\end{conj}

This conjecture seems much more plausible than its negation, which is equivalent to the existence of a moonshine-like vertex operator algebra that admits an automorphism of order 9 with trace $f_{9a}$, but no non-Fricke automorphisms.

The function 9a is the cube root of $T_{3A}(3\tau) + 42 = q^{-3} + 42 + 783q^3 + 8672q^6 + 65367q^9 + \cdots$. From \cite{F93}, we know that its eigengroup has the form
\begin{gather*} \Gamma_0(9|3)^+ = \left( \begin{smallmatrix} 1/3 & 0 \\ 0 & 1 \end{smallmatrix} \right) \Gamma_0(3)^+ \left( \begin{smallmatrix} 3 & 0 \\ 0 & 1 \end{smallmatrix} \right), \end{gather*}
and the fixing group is the index 3 normal subgroup uniquely determined by the fact that $y = \left( \begin{smallmatrix} 1 & 1/3 \\ 0 & 1 \end{smallmatrix} \right)$ acts by $e(-1/3)$, and $x = \left( \begin{smallmatrix} 1 & 0 \\ 9 & 1 \end{smallmatrix} \right)$ acts by $e(1/3)$. The fixing group contains $\Gamma_0(27)$ as an index 6 subgroup. In particular, the fixing group lies in $N_{{\rm SL}_2(\bR)}(\Gamma_0(27))$, but not in the involutory normalizer. This failure to lie in the involutory normalizer is precisely the property that eliminates this group from the Conway--Norton list of candidate functions. Unfortunately, the involutory normalizer of $\Gamma_0(N)$ does not seem to have a vertex-operator-algebraic interpretation at this time.

An extensive computation strongly suggests that all coefficients of the vector-valued modular function $F^{9a}$, constructed in Section~3 of~\cite{GM2}, are non-negative integers. The same seems to be true for 63a and 117a. This means the existence of abelian intertwining algebras attached to cyclic orbifolds (shown in~\cite{vEMS}) does not yield a criterion that we can use to show that these functions are non-moonshine-like.

Naturally, we would be very interested in any new conditions on trace functions that could lead to a resolution to this conjecture.

\subsection*{Acknowledgements}
This research was funded by JSPS Kakenhi Grant-in-Aid for Young Scientists (B) 17K14152.

\pdfbookmark[1]{References}{ref}
\LastPageEnding

\end{document}